\documentclass[12pt]{amsart}      
\usepackage{txfonts}
\usepackage{amssymb}
\usepackage{eucal}
\usepackage{graphicx}
\usepackage{amsmath}
\usepackage{amscd}
\usepackage[all]{xy}           
\usepackage[active]{srcltx} 

\usepackage{amsfonts,latexsym}
\usepackage{xspace}
\usepackage{epsfig}
\usepackage{float}
\usepackage{color}
\usepackage{fancybox}
\usepackage{colordvi}
\usepackage{multicol}
\usepackage{colordvi}
\usepackage[hypertex]{hyperref} 

\newtheorem{theorem}{Theorem}[section]
\newtheorem{proposition}[theorem]{Proposition}
\newtheorem{definition}[theorem]{Definition}
\newtheorem{lemma}[theorem]{Lemma}
\newtheorem{corollary}[theorem]{Corollary}
\newtheorem{prop-def}{Proposition-Definition}[section]
\newtheorem{coro-def}{Corollary-Definition}[section]

\newtheorem{remark}[theorem]{Remark}

\newtheorem{problem}[theorem]{Problem}

\newcommand{\nc}{\newcommand}
\nc{\tred}[1]{\textcolor{red}{#1}}
\nc{\tblue}[1]{\textcolor{blue}{#1}}
\nc{\tgreen}[1]{\textcolor{green}{#1}}
\nc{\tpurple}[1]{\textcolor{purple}{#1}}
\nc{\btred}[1]{\textcolor{red}{\bf #1}}
\nc{\btblue}[1]{\textcolor{blue}{\bf #1}}
\nc{\btgreen}[1]{\textcolor{green}{\bf #1}}
\nc{\btpurple}[1]{\textcolor{purple}{\bf #1}}

\renewcommand{\Bbb}{\mathbb}

\newcommand{\efootnote}[1]{}

\renewcommand{\textbf}[1]{}

\newcommand{\delete}[1]{}
\nc{\dfootnote}[1]{{}}          
\nc{\ffootnote}[1]{\dfootnote{#1}}
\nc{\mfootnote}[1]{\footnote{#1}} 
\nc{\ofootnote}[1]{\footnote{\tiny Older version: #1}}

\nc{\mlabel}[1]{\label{#1}}  
\nc{\mcite}[1]{\cite{#1}}  
\nc{\mref}[1]{\ref{#1}}  
\nc{\mbibitem}[1]{\bibitem{#1}} 

\delete{
\nc{\mlabel}[1]{\label{#1}  
{\hfill \hspace{1cm}{\bf{{\ }\hfill(#1)}}}}
\nc{\mcite}[1]{\cite{#1}{{\bf{{\ }(#1)}}}}  
\nc{\mref}[1]{\ref{#1}{{\bf{{\ }(#1)}}}}  
\nc{\mbibitem}[1]{\bibitem[\bf #1]{#1}} 
}

\nc{\mtail}{\leq_t}
\nc{\mhead}{\leq_h}
\nc{\rk}{\mathrm{rk}}
\nc{\mset}[1]{\tilde{#1}}
\nc{\pa}{\frakL}
\nc{\arr}{\rightarrow}
\nc{\lu}[1]{(#1)}
\nc{\mult}{\mrm{mult}}
\nc{\diff}{\mathrm{Der}}
\nc{\indiff}{\mathrm{InDer}}
\nc{\outdiff}{\mathrm{OutDer}}
\nc{\conmat}{connection matrix\xspace}
\nc{\bounmat}{boundary matrix\xspace}
\nc{\pcyc}{\mathfrak c}
\nc{\calpa}{\calp_A}
\nc{\calpal}{\Gamma_{AL}}
\nc{\calpc}{\calp_L}
\nc{\frakDa}{\frakD_1}
\nc{\frakDal}{\frakD_2}
\nc{\frakDc}{\frakD_L}
\nc{\frakDv}{\frakD_V}
\nc{\frakDp}{\frakD_F}
\nc{\frakBa}{\frakB_1}
\nc{\frakBal}{\frakB_2}
\nc{\frakBc}{\frakB_L}
\nc{\frakBv}{\frakB_V}

\nc{\bin}[2]{ (_{\stackrel{\scs{#1}}{\scs{#2}}})}  
\nc{\binc}[2]{ \left (\!\! \begin{array}{c} \scs{#1}\\
    \scs{#2} \end{array}\!\! \right )}  
\nc{\bincc}[2]{  \left ( {\scs{#1} \atop
    \vspace{-1cm}\scs{#2}} \right )}  
\nc{\bs}{\bar{S}}
\nc{\cosum}{\sqsubset}
\nc{\la}{\longrightarrow}
\nc{\rar}{\rightarrow}
\nc{\dar}{\downarrow}
\nc{\dprod}{**}
\nc{\dap}[1]{\downarrow \rlap{$\scriptstyle{#1}$}}
\nc{\md}{\mathrm{dth}}
\nc{\uap}[1]{\uparrow \rlap{$\scriptstyle{#1}$}}
\nc{\defeq}{\stackrel{\rm def}{=}}
\nc{\disp}[1]{\displaystyle{#1}}
\nc{\dotcup}{\ \displaystyle{\bigcup^\bullet}\ }
\nc{\gzeta}{\bar{\zeta}}
\nc{\hcm}{\ \hat{,}\ }
\nc{\hts}{\hat{\otimes}}
\nc{\barot}{{\otimes}}
\nc{\free}[1]{\bar{#1}}
\nc{\uni}[1]{\tilde{#1}}
\nc{\hcirc}{\hat{\circ}}
\nc{\lleft}{[}
\nc{\lright}{]}
\nc{\lc}{\lfloor}
\nc{\rc}{\rfloor}
\nc{\curlyl}{\left \{ \begin{array}{c} {} \\ {} \end{array}
    \right .  \!\!\!\!\!\!\!}
\nc{\curlyr}{ \!\!\!\!\!\!\!
    \left . \begin{array}{c} {} \\ {} \end{array}
    \right \} }
\nc{\longmid}{\left | \begin{array}{c} {} \\ {} \end{array}
    \right . \!\!\!\!\!\!\!}
\nc{\onetree}{\bullet}
\nc{\ora}[1]{\stackrel{#1}{\rar}}
\nc{\ola}[1]{\stackrel{#1}{\la}}
\nc{\ot}{\otimes}
\nc{\mot}{{{\boxtimes\,}}}
\nc{\otm}{\overline{\boxtimes}}
\nc{\sprod}{\bullet}
\nc{\scs}[1]{\scriptstyle{#1}}
\nc{\mrm}[1]{{\rm #1}}
\nc{\margin}[1]{\marginpar{\rm #1}}   
\nc{\dirlim}{\displaystyle{\lim_{\longrightarrow}}\,}
\nc{\invlim}{\displaystyle{\lim_{\longleftarrow}}\,}
\nc{\mvp}{\vspace{0.3cm}}
\nc{\tk}{^{(k)}}
\nc{\tp}{^\prime}
\nc{\ttp}{^{\prime\prime}}
\nc{\svp}{\vspace{2cm}}
\nc{\vp}{\vspace{8cm}}
\nc{\proofbegin}{\noindent{\bf Proof: }}
\nc{\proofend}{$\blacksquare$ \vspace{0.3cm}}
\nc{\modg}[1]{\!<\!\!{#1}\!\!>}
\nc{\intg}[1]{F_C(#1)}
\nc{\lmodg}{\!<\!\!}
\nc{\rmodg}{\!\!>\!}
\nc{\cpi}{\widehat{\Pi}}
\nc{\sha}{{\mbox{\cyr X}}}  
\nc{\shap}{{\mbox{\cyrs X}}} 
\nc{\shpr}{\diamond}    
\nc{\shp}{\ast}
\nc{\shplus}{\shpr^+}
\nc{\shprc}{\shpr_c}    
\nc{\msh}{\ast}
\nc{\zprod}{m_0}
\nc{\oprod}{m_1}
\nc{\vep}{\varepsilon}
\nc{\labs}{\mid\!}
\nc{\rabs}{\!\mid}

\nc{\mmbox}[1]{\mbox{\ #1\ }}
\nc{\fp}{\mrm{FP}} \nc{\rchar}{\mrm{char}} \nc{\End}{\mrm{End}} \nc{\Fil}{\mrm{Fil}}
\nc{\Mor}{Mor\xspace}
\nc{\gmzvs}{gMZV\xspace}
\nc{\gmzv}{gMZV\xspace}
\nc{\mzv}{MZV\xspace}
\nc{\mzvs}{MZVs\xspace}
\nc{\Hom}{\mrm{Hom}} \nc{\id}{\mrm{id}} \nc{\im}{\mrm{im}}
\nc{\incl}{\mrm{incl}} \nc{\map}{\mrm{Map}} \nc{\mchar}{\rm char}
\nc{\nz}{\rm NZ} \nc{\supp}{\mathrm Supp}

\nc{\Alg}{\mathbf{Alg}}
\nc{\Bax}{\mathbf{Bax}}
\nc{\bff}{\mathbf f}
\nc{\bfk}{{\bf k}}
\nc{\bfone}{{\bf 1}}
\nc{\bfx}{\mathbf x}
\nc{\bfy}{\mathbf y}
\nc{\base}[1]{\bfone^{\otimes ({#1}+1)}} 
\nc{\Cat}{\mathbf{Cat}}

\nc{\detail}{\marginpar{\bf More detail}
    \noindent{\bf Need more detail!}
    \svp}
\nc{\Int}{\mathbf{Int}}
\nc{\Mon}{\mathbf{Mon}}
\nc{\rbtm}{{shuffle }}
\nc{\rbto}{{Rota-Baxter }}
\nc{\remarks}{\noindent{\bf Remarks: }}
\nc{\Rings}{\mathbf{Rings}}
\nc{\Sets}{\mathbf{Sets}}

\nc{\BA}{{\Bbb A}} \nc{\CC}{{\Bbb C}} \nc{\DD}{{\Bbb D}}
\nc{\EE}{{\Bbb E}} \nc{\FF}{{\Bbb F}} \nc{\GG}{{\Bbb G}}
\nc{\HH}{{\Bbb H}} \nc{\LL}{{\Bbb L}} \nc{\NN}{{\Bbb N}}
\nc{\KK}{{\Bbb K}} \nc{\QQ}{{\Bbb Q}} \nc{\RR}{{\Bbb R}}
\nc{\TT}{{\Bbb T}} \nc{\VV}{{\Bbb V}} \nc{\ZZ}{{\Bbb Z}}

\nc{\cala}{{\mathcal A}} \nc{\calc}{{\mathcal C}}
\nc{\cald}{{\mathcal D}} \nc{\cale}{{\mathcal E}}
\nc{\calf}{{\mathcal F}} \nc{\calg}{{\mathcal G}}
\nc{\calh}{{\mathcal H}} \nc{\cali}{{\mathcal I}}
\nc{\call}{{\mathcal L}} \nc{\calm}{{\mathcal M}}
\nc{\caln}{{\mathcal N}} \nc{\calo}{{\mathcal O}}
\nc{\calp}{{\mathcal P}} \nc{\calr}{{\mathcal R}}
\nc{\cals}{{\mathcal S}}
\nc{\calt}{{\mathcal T}} \nc{\calw}{{\mathcal W}}
\nc{\calk}{{\mathcal K}} \nc{\calx}{{\mathcal X}}
\nc{\CA}{\mathcal{A}}

\nc{\fraka}{{\mathfrak a}}
\nc{\frakA}{{\mathfrak A}}
\nc{\frakb}{{\mathfrak b}}
\nc{\frakB}{{\mathfrak B}}
\nc{\frakC}{{\mathfrak C}}
\nc{\frakD}{{\mathfrak D}}
\nc{\frakg}{{\mathfrak g}}
\nc{\frakH}{{\mathfrak H}}
\nc{\frakL}{{\mathfrak L}}
\nc{\frakM}{{\mathfrak M}}
\nc{\bfrakM}{\overline{\frakM}}
\nc{\frakm}{{\mathfrak m}}
\nc{\frakP}{{\mathfrak P}}
\nc{\frakN}{{\mathfrak N}}
\nc{\frakp}{{\mathfrak p}}
\nc{\frakR}{{\mathfrak R}}
\nc{\frakS}{{\mathfrak S}}

\font\cyr=wncyr10
\font\cyrs=wncyr7
\nc{\li}[1]{\textcolor{red}{Li:#1}}
\nc{\yang}[1]{\textcolor{blue}{Yang: #1}}
\nc{\liu}[1]{\textcolor{green}{Liu: #1}}
\nc{\del}[1]{\textcolor{red}{Del: #1}}

\begin{document}

\title[Genuses of cluster quivers of finite mutation type]
{Genuses of cluster quivers of finite mutation type}
%
%
\author[Fang Li]{Fang Li}
\address{Department of Mathematics, Zhejiang University, Hangzhou 310027, China}
\email{fangli@zju.edu.cn}
\author[Jichun Liu]{Jichun Liu$^{*}$}
\address{Department of Mathematics, Zhejiang University, Hangzhou 310027, China}
\email{liujc1982@126.com}
\author[Yichao Yang]{Yichao Yang}
\address{Department of Mathematics, Zhejiang University, Hangzhou 310027, China}
\email{yyc880113@163.com}


\renewcommand{\thefootnote}{\alph{footnote}}
\setcounter{footnote}{-1} \footnote{}
\renewcommand{\thefootnote}{\alph{footnote}}
\setcounter{footnote}{-1} \footnote{\emph{2010 Mathematics Subject
Classification}: 05C10; 13F60; 05E10; 05E15.}

\renewcommand{\thefootnote}{\alph{footnote}}
\setcounter{footnote}{-1} \footnote{\emph{Keywords}: cluster quiver;
finite mutation type; genus; triangulation of surface.}

\renewcommand{\thefootnote}{\alph{footnote}}
\setcounter{footnote}{-1} \footnote{*: Corresponding author}

\begin{abstract}
In this paper, we study the  distribution of the genuses of cluster
quivers of finite mutation type. First, we prove that in the $11$
exceptional cases, the distribution of genuses is $0$ or $1$. Next,
we consider the relationship between the genus of an oriented
surface and that of cluster quivers from this surface. It is
verified that the genus of an oriented surface is an upper bound for
the genuses of cluster quivers from this surface. Furthermore, for
any non-negative integer $n$ and a closed oriented surface of genus
$n$, we show that there always exist a set of punctures and a
triangulation of this surface such that the corresponding cluster
quiver from this triangulation is exactly of genus $n$.
\end{abstract}

\maketitle

\tableofcontents

\setcounter{section}{0}

\section{Introduction}

Cluster quiver is a valuable notion in the theory of cluster algebras.
Cluster algebras were introduced by Fomin and Zelevinsky in the
famous paper \cite{fz1}. Since then this subject has been studied
extensively by many mathematicians. The original motivation was to
give a combinatorial characterization of dual canonical bases in the theory of quantum groups, and for the study of total positivity for
algebraic groups. Now cluster algebras are connected to various fields of
mathematics such as representation theory, Poisson geometry,
algebraic geometry, Lie theory, combinatorics and so on.
One knows that cluster algebras are commutative algebras equipped with
a distinguished set of generators, i.e.,
cluster variables.

Two types of cluster algebras are of important interests: finite
type and finite mutation type. The former is a special case of the
latter. Cluster algebras of finite type have been completely
classified in \cite{fz2} and skew-symmetric cluster algebras of
finite mutation type have been completely classified in \cite{fst}.
The classification of cluster algebras of finite type is identical
to the Cartan-Killing classification of semisimple Lie algebras and
finite root systems. For a cluster algebra of finite type, there is
a one-one correspondence between the set of cluster variables and
that of almost positive roots (consisting of positive roots and
negative simple roots). Additionally, the classification of
skew-symmetric cluster algebras (equivalently, the classification of
cluster quivers) of finite mutation type tells us that almost all
skew-symmetric cluster algebras (equivalently, cluster quivers) of
this type come from triangulations of surfaces except for $11$
exceptional cases.

Given an oriented 2-dimensional Riemann surface $S$ with boundary
$\partial{S}$, let $M\subset S$ be a finite set of {\bf marked
points} such that each connected boundary component contains at
least one such point. Marked points in the interior of $S$ are
called {\bf punctures}. The pair $(S, M)$ is simply called a {\bf
surface}. An {\bf arc}(\cite{fosth}) is the homotopy class of a
curve $\gamma$ in $S$ whose endpoints come from $M$ such that

$\bullet$ $\gamma$ does not intersect itself, except that its
endpoints may coincide;

$\bullet$ except for the endpoints, $\gamma$ is disjoint from $M$
and $\partial{S}$;

$\bullet$ $\gamma$ does not cut out an unpunctured monogon or an
unpunctured digon.

An {\bf ideal triangulation} $T$ is a maximal set of non-crossing
(i.e., there are no intersections in the interior of $S$) arcs. For
the details of the construction of cluster quivers from
triangulations of surfaces, see section 2.2.

{\em In this paper, all surfaces we consider are oriented surfaces;
all subgraphs and subquivers are full.}

In topological graph theory, the {\bf genus} of a graph is the
minimal genus of surfaces where the graph can be drawn without
crossing. The {\bf genus} of a quiver is defined to be that of its
underlying graph. When discussing the genus of a quiver, one only
needs to consider its simple underlying graph (without
multiple edges and orientation). A graph (resp.
quiver) is {\bf planar} if it is of genus $0$. It is well known that
genus is a topological invariant for surfaces, as well as for
topological graphs. A natural question is  to find out the relation
between the genus of a surface and that of a cluster quiver from
this surface.
 As an answer, we have the main conclusion in this paper as follows:
\begin{theorem}\label{thm1.1}
(i)~ For a triangulation $T$ of a surface $S$ with genus $g$, let $g'$ be
the genus of the cluster quiver $Q$ associated with $T$. Then,
$g'\leq g$.

(ii)~ Furthermore, for any non-negative integer $n$ and a closed
oriented surface $S_{n}$ of genus $n$, there exists a set of marked
points $M$ on $S_{n}$ and an ideal triangulation $P_{n}$ of $S_{n}$
such that the corresponding cluster quiver $T_{n}$ of $P_{n}$ has
genus $n$.
\end{theorem}
From this result, we know that the genus of a
surface is in fact an upper bound for the genuses  of cluster
quivers from the triangulations of this surface
and moreover any non-negative integer $n$ can be
reached as the genus of some cluster quiver from surface.

The paper is organized as follows. The requisite backgrounds on
cluster quiver and its mutation, triangulation of surface are
presented in $\S$ 2. In $\S$ 2.1, we give the basic definitions of
matrix mutation and quiver mutation. We mention the fact that skew-symmetric matrices
are in bijection with cluster quivers, also that
matrix mutation and quiver mutation are compatible.  In $\S$ 2.2, we
recall some basic definitions and properties of triangulation of
surface from \cite{fosth}. One can see how to obtain a cluster
quiver from a triangulation of surface and the compatibility between
mutation of quiver and flip of triangulation. A cluster quiver comes
from surface if and only if it is block-decomposable. As the end of
this subsection, we restate the classification of skew-symmetric
cluster algebras of finite mutation type.

$\S$ 3 mainly deals with the genuses of cluster quivers of finite
mutation type. In $\S$ 3.1, we give the table of genus distribution
of the $11$ exceptional quivers by utilizing Keller's quiver
mutation in Java \cite{k}. In $\S$ 3.2, we first prove Theorem
\ref{thm1.1}(i) which states that the genus of a surface is an upper
bound of genuses of cluster quivers obtained by triangulations of
this surface. From this result, one can easily see that genus is a
mutation invariant for cluster quivers from the surface of genus
$0$. As another application of this result, we give a sufficient
condition for two quivers  not to be mutation equivalent. The part
(ii) of Theorem \ref{thm1.1} is proved through constructing the
graph $R_{n}$, using topological graph theory for genus $n$ and the
classification theorem of compact surfaces in algebraic topology.

\section{Preliminaries}

\subsection{Cluster quiver and its mutation}
The notion of skew-symmetric matrix or equivalently of cluster
quiver is crucial in the theory of cluster algebras. In the
definition of cluster algebras, the most important ingredient is the
so-called {\em seed mutation}. For our purpose in this paper, we
only introduce {\em matrix mutation} (an important part of seed
mutation) so as to understand the motivation of
cluster quivers. For the details of the definitions of seed mutation
and cluster algebra, we refer to \cite{fz2}.

Suppose $B=(b_{ij})$ is an $n\times n$ integer matrix. For $1\leq
k\leq n$, a {\bf matrix mutation} $\mu_{k}$ at direction $k$
transforms $B$ into a new matrix $B^{'}=(b_{ij}^{'})$ where
$b_{ij}^{'}$ is defined by the equation
\begin{eqnarray*}b'_{ij}=\left\{\begin{array}{lll} -b_{ij},& if
~~i=k \ or j=k;\\b_{ij}+\frac{|b_{ik}|b_{kj}+b_{ik}|b_{kj}|}{2},&
otherwise.\end{array}\right.\end{eqnarray*}

Here, all matrices we consider are {\em skew-symmetric}. It is
easy to see that matrix mutation transforms a skew-symmetric matrix
into another one.

Given an $n\times n$ skew-symmetric matrix
$B=(b_{ij})$, we can construct a quiver $Q$ without loops and
$2$-cycles as follows: the vertex set is just the row/column indexes
$1,2, \cdots, n$ of the matrix $B$ and the number of arrows from $i$
to $j$ is defined to be $b_{ij}$ if $b_{ij}>0$.

\begin{definition}
A quiver without
loops and $2$-cycles is said to be a {\bf cluster quiver}.
\end{definition}
 There is a one-one correspondence between the set of skew-symmetric
matrices and that of cluster quivers. In fact, given a cluster
quiver $Q$ with $n$ vertices, one can construct a skew-symmetric
matrix $B=(b_{ij})$ defined by $b_{ij}=\sharp\{i\rightarrow
j\}-\sharp\{j\rightarrow i\}$ where $\sharp\{i\rightarrow j\}$
denotes the number of arrows from $i$ to $j$. According to this
one-one correspondence, {\em quiver mutation} can be deduced from
matrix mutation.

\begin{definition}
Suppose $Q$ is a cluster quiver with $n$ vertices where $Q_{0}=\{1,
2,\cdots, n\}$. For $k\in Q_{0}$, a {\bf quiver mutation} $\mu_{k}$
at vertex $k$ transforms $Q$ into $Q'$ where $Q'$ is obtained by the
following three steps:

(1)~ For every path $i\rightarrow k\rightarrow j$, add a new arrow
$i\rightarrow j$;

(2)~ Reverse all arrows incident with $k$;

(3)~ Delete all $2$-cycles.

\end{definition}

One can easily see that the resulting quiver
$Q'$ is also a cluster quiver. Matrix mutation and quiver mutation
are compatible in the following sense: given any $k\in {1, 2,\cdots,
n}$, $\mu_{k}(Q_B)=Q_{\mu_{k}(B)}$ and
$\mu_{k}(B_Q)=B_{\mu_{k}(Q)}$.

It is easy to verify that both matrix mutation and quiver
mutation are involutions, i.e., $\mu_{k}^{2}=1$. If
$Q'=\mu_{k_{1}}\mu_{k_{2}}\cdots \mu_{k_{l}}(Q)$ for some $k_{1},
k_{2}\cdots k_{l} \in \{1, 2, \cdots, n\}$, we will say $Q$ and $Q'$
are {\bf mutation equivalent}. Obviously it is an equivalence
relation on the set of isomorphism classes of cluster quivers with
$n$ vertices. A cluster quiver( resp. skew-symmetric cluster algebra
constructed from this quiver) is said to be of {\bf finite mutation
type} if the number of quivers in its mutation equivalent class is
finite. Cluster quivers of this type have been completely classified
in \cite{fst}. We will restate this classification theorem in $\S
2.2$.

\subsection{Cluster quivers from surface}
Given a surface $(S, M)$, the number of arcs in any triangulations
of $(S, M)$ is a constant. The following lemma gives the formula to
calculate the number of arcs in a triangulation.

\begin{lemma}\cite{fosth}~ For a triangulation of a surface, the
following formula holds:
\begin{equation}\label{relation1}
n= 6g+ 3b+ 3p+ c- 6
\end{equation}
 where
$n$ is the number of arcs; $g$ is the genus of the surface; $b$ is
the number of connected boundary components; $p$ is the number of
punctures; $c$ is the number of marked points on the boundary.
\end{lemma}

The arcs of an ideal triangulation cut the surface $S$ into {\bf
ideal triangles}. The three sides of an ideal triangle do not have
to be distinct, i.e., we allow {\bf self-folded} triangle which can
be shown as the following figure:

\begin{figure}[h] \centering

  \includegraphics*[217,473][293,541]{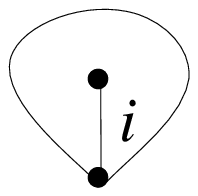}

 Figure 1
\end{figure}

Given an ideal triangulation $T$, there is an associated signed
adjacency matrix $B(T)$( see \cite{fosth} $\S 4$). Suppose the arcs
in $T$ are labeled by the numbers $1, 2, \cdots, n$ and let the rows and columns of $B(T)$ be
 numbered from $1$ to $n$. For an arc $i$, let $\pi_{T}(i)$ denote
the arc defined as follows: if there is a self-folded ideal triangle
in $T$ folded along $i$ (see Figure $1$), then $\pi_{T}(i)$ is its
remaining side; otherwise, we set $\pi_{T}(i)=i$.

For each non-self-folded triangle $\triangle$, define the $n\times
n$ integer matrix $B^{\triangle}=(b_{ij}^{\triangle})$ by setting

\begin{eqnarray*}b_{ij}^{\triangle}=\left\{\begin{array}{lll} 1,& \text{if} ~\triangle~ \text{has sides labeled}
~\pi_{T}(i)~ \text{and}~ \pi_{T}(j)~ \text{with}~ \pi_{T}(j)~ \text{following}~
\pi_{T}(i) ~\text{in the clockwise order;} \\ -1, & \text{if the
same holds, with the counter-clockwise order;}
\\0,& \text{ otherwise.}\end{array}\right.\end{eqnarray*}

The matrix $B=B(T)=(b_{ij})$ is defined by
\begin{equation}
B=\sum_{\triangle}B^{\triangle}
\end{equation}
where the sum is taken over all non-self-folded triangle $\triangle$. It is easy
to verify that $B(T)$ is skew-symmetric, and all its entries are
equal to $0$, $1$, $-1$, $2$ or $-2$.
Therefore, given a
triangulation $T$, we can firstly associate a skew-symmetric matrix
$B(T)$ to $T$ and then obtain a cluster quiver $Q$ corresponding to
$B(T)$, just as that given in \S 2.1.
 The corresponding cluster
quiver $Q_B$ of $B=B(T)$ is said to be {\bf coming from  surface}. Correspondingly,
the cluster algebra defined by $Q_B$ is also said to be {\em coming
from surface}.

A {\bf flip} is a transformation of an ideal triangulation $T$ into
a new triangulation $T^{'}$ obtained by replacing an arc $\gamma$
with a unique different arc $\gamma^{'}$ and leaving other arcs
unchanged. Flip of triangulation and mutation of matrix are
compatible in the sense of the following proposition.

\begin{proposition}(\cite{fosth} Proposition 4.8)
Suppose that the triangulation $\overline{T}$ is obtained from $T$
by a flip replacing an arc $k$. Then
$B(\overline{T})=\mu_{k}(B(T))$.

\end{proposition}

According to the Remark 4.2 in \cite{fosth}, all triangulations that
we are interested in can be obtained by gluing together a number of
puzzle pieces except for one case, i.e. the triangulation of 4-punctured
sphere obtained by gluing three self-folded triangles to respective
sides of an ordinary triangle, see Figure 2(I). There are three types of puzzle pieces as follows, see
Figure 2(II).

\begin{figure}[h] \centering

  \includegraphics*[229,393][315,465]{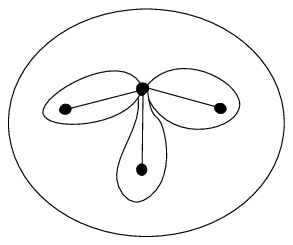}

Figure 2(I)
\end{figure}

\begin{figure}[h] \centering

  \includegraphics*[124,475][432,545]{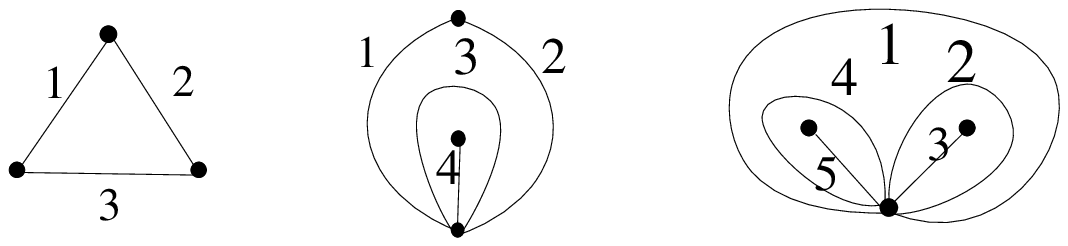}

Figure 2(II)
\end{figure}
These three types of puzzle pieces correspond to blocks of type I-V
(see Figure 3) depending on whether the outer sides are
lying on the boundary (for the details, see the proof of Theorem
13.3 in \cite{fosth}).
\begin{figure}[h] \centering

  \includegraphics*[59,422][344,481]{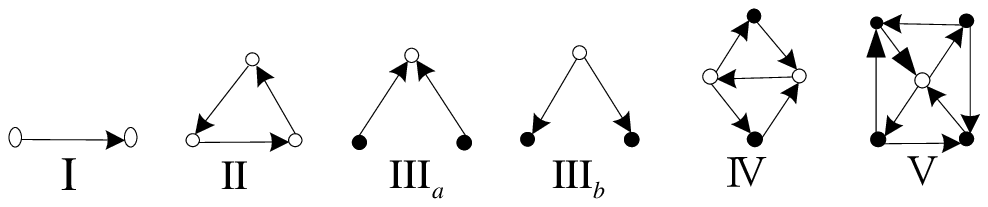}

 Figure 3
\end{figure}
The vertices marked by unfilled circles in Figure 3 are called {\bf
outlets}.
\begin{definition}(\cite{fosth})
A quiver is said to be {\bf block-decomposable} if it can be
obtained from a collection of disjoint blocks by the following
procedure:

(1) take a partial matching of the combined set of outlets (matching
an outlet to itself or to another outlet from the same block is not
allowed);

(2) glue the outlets in each pair of matching;

(3) remove all 2-cycles.

\end{definition}

According to Theorem 13.3 in \cite{fosth}, a cluster quiver is coming from
surface if and only if it is block-decomposable.

The following theorem gives a complete classification of
skew-symmetric cluster algebras of finite mutation type.

\begin{lemma}\label{lem2.1}\cite{fst}.
A skew-symmetric cluster algebra $\mathcal{A}$ of rank $n$ is of
finite mutation type if and only if:

(1) $\mathcal{A}$ is coming from surface ($n\geq {3}$);\\
or~ (2) $n\leq {2}$;\\
or~ (3) $\mathcal{A}$ is one of the $11$ exceptional types (see
Figure 4): $$E_{6},~ E_{7},~ E_{8},~ E^{(1)}_{6},~ E^{(1)}_{7},~
E^{(1)}_{8},~ E^{(1,1)}_{6},~ E^{(1,1)}_{7},~ E^{(1,1)}_{8},~
X_{6},~ X_{7}$$ that is, $\mathcal{A}$ has a cluster quiver at one
of the above types.
\begin{figure}[h] \centering

  \includegraphics*[115,474][428,725]{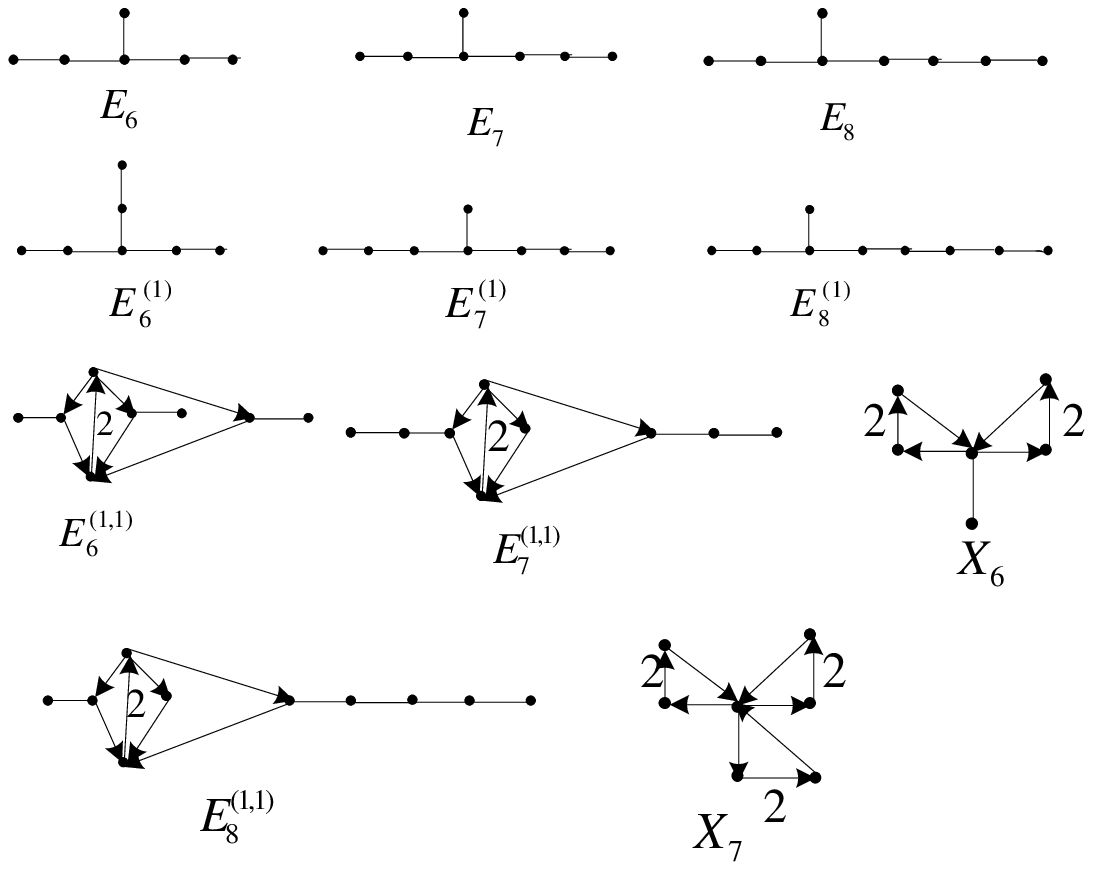}

Figure 4
\end{figure}

\end{lemma}

\section{Distribution of genuses of cluster quivers of finite mutation type}

\subsection{Genuses of exceptional cluster quivers}

In this section, we give the table of distribution of genuses of
the $11$ exceptional cluster quivers in the classification of finite mutation
type. Our main tool is Keller's quiver mutation in java(\cite{k}). To
obtain the following table, one should note the following facts:

(1)~ $E_{6},~ E_{7},~ E_{8},~ E^{(1)}_{6},~ E^{(1)}_{7}$ and
$E^{(1)}_{8}$ are trees. According to Lemma 1.1 in \cite{V}, any
orientations on the same tree are mutation equivalent.

(2)~ $E_{6},~ E_{7},~ E_{8}$ and $E^{(1)}_{8}$ are full subgraphs of
the underlying graph of  $E^{(1,1)}_{8}$; $E^{(1)}_{6}$ is a full
subgraph of the underlying graph of $E^{(1,1)}_{6}$; $E^{(1)}_{7}$
is a full subgraph of the underlying graph of $E^{(1,1)}_{7}$. Since
any quiver mutation equivalent to a full subquiver of $Q$ must be a
full subquiver of some $Q^{'}$ which is mutation equivalent to $Q$,
we first test the mutation classes of
$E^{(1,1)}_{6}$, $E^{(1,1)}_{7}$ and $E^{(1,1)}_{8}$ in order to see
their genus distribution.

(3)~ To see the genus of a quiver, we only need to see its
underlying graph. Hence when doing the quiver mutation in java
 due to \cite k, we can choose the mutation class under graph
isomorphism. This can cut down the number of quivers in mutation
class greatly which we have to consider.

(4) We check all the quivers in the mutation classes of
$E^{(1,1)}_{6}$, $E^{(1,1)}_{7}$ and $E^{(1,1)}_{8}$ and find that all of them are planar. So are the other exceptional cluster quivers of $E$ type.
\\

In the following distribution table, the {\em total number} means the
number of quivers in the mutation class up to quiver isomorphism,
and the {\em number of genus} $0$ (resp. $1$) means the number of quivers (up to
quiver isomorphism) in the mutation class whose genus is $0$ (resp. $1$).

\begin{tabular}{|c|c|c|c|}
  \hline
  $11$ exceptional cluster quivers & total number & the number of genus 0 & the number of genus 1 \\
\hline
  $E_{6}$ & $21$ & $21$ & $0$ \\
\hline
  $E_{7}$ & $112$ & $112$ & $0$ \\
\hline
  $E_{8}$ & $391$ & $391$ & $0$ \\
\hline
  $E_{6}^{(1)}$ & $52$ & $52$ & $0$ \\
\hline
  $E_{7}^{(1)}$ & $338$ & $338$ & $0$ \\
\hline
  $E_{8}^{(1)}$ & $1935$ & $1935$ & $0$ \\
\hline
  $E_{6}^{(1,1)}$ & $27$ & $27$ & $0$ \\
\hline
  $E_{7}^{(1,1)}$ & $217$ & $217$ & $0$ \\
\hline
  $E_{8}^{(1,1)}$ & $1886$ & $1886$ & $0$ \\
\hline
  $X_{6}$ & $4$ & $1$ & $3$ \\
\hline
  $X_{7}$ & $2$ & $1$ & $1$ \\
  \hline
\end{tabular}
\\
\\
 From the above table, one can easily see that the
genus of the quiver of type $E$ is invariant under quiver mutation; but the genus of the quiver of type $X$ will vary under
quiver mutation.

\begin{proposition}
There are exactly $4$ non-planar cluster quivers of exceptional finite mutation types, with genuses $1$, which are listed as follows,
\begin{figure}[h] \centering

  \includegraphics*[83,431][486,543]{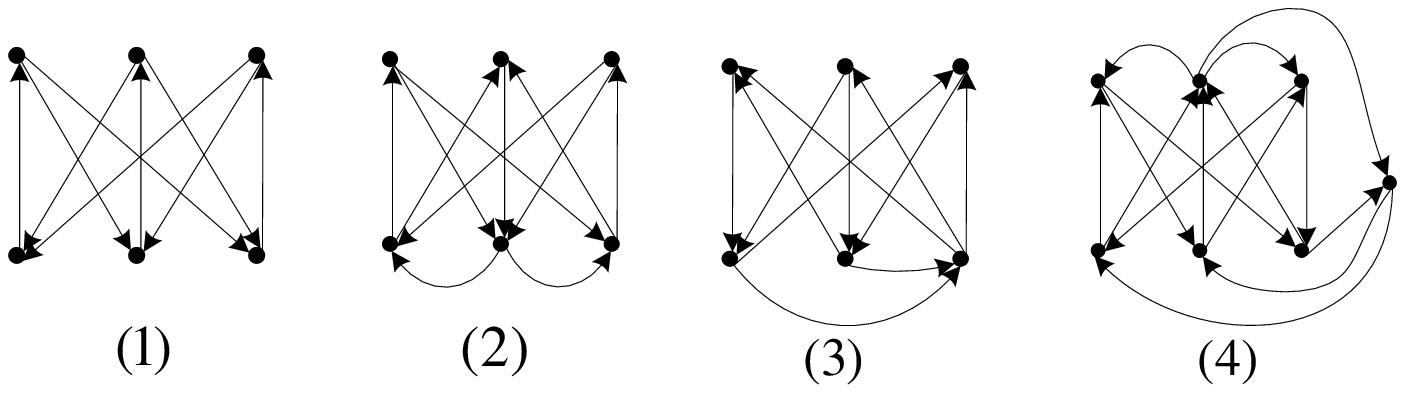}

Figure 5
\end{figure}
where the quivers (1), (2), (3) are in the mutation-equivalent class of $X_{6}$, the
quiver (4) is in the
mutation-equivalent class of $X_{7}$.
\end{proposition}

\begin{proof}
For the convenience of describing the
mutations at $X_{6}$ and $X_{7}$ to obtain these four quivers, we
will label the vertices of $X_{6}$ and $X_{7}$ as Figure 6.
\begin{figure}[h] \centering

  \includegraphics*[171,392][362,475]{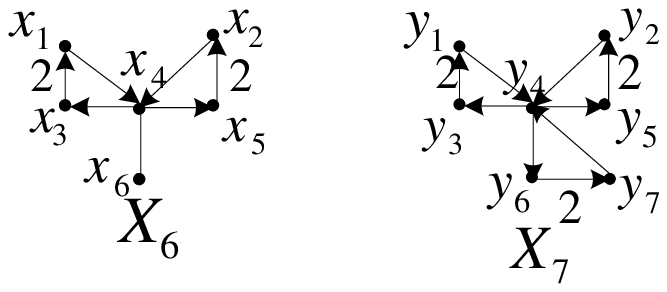}

Figure 6
\end{figure}
Then, we can get the four quivers in Figure 5 as follows:

the quiver (1) can be obtained from $X_6$ by mutation on the vertices $x_{4}$ and $x_{6}$;

the quiver (2) can be obtained from $X_6$ by mutation on the vertex
$x_{4}$;

the quiver (3) can be obtained from $X_6$ by mutation on the vertices $x_{4}$ and $x_{3}$;

the quiver (4) can be obtained from $X_7$ by mutation on the vertex
$y_{4}$.
\end{proof}

\subsection{Proof of the main conclusion}

We will begin by proving the first part of the
theorem, i.e., the genuses of cluster quivers obtained from the
triangulations of a surface are not greater than that of the
surface.

{\em Proof of Theorem \ref{thm1.1}}(i).~ By the correspondence of
puzzle pieces and blocks, each puzzle piece corresponds to a block
of types I-V. For each puzzle piece, we put its corresponding block
into the face bounded by it. If two puzzle pieces have a common
edge, then we glue two vertices corresponding to the common edge
between these two blocks. Hence we obtain the quiver $Q$ of T in
this way and moreover the underlying graph of $Q$ can be drawn
without self-crossing on the surface $S$. We then have $g'\leq g$ by
the definition of genus of quiver.

To complete the proof of the theorem, we should consider the only
exceptional case the triangulation of
which cannot be obtained by gluing the puzzle pieces. Let T be the
triangulation of 4-punctured sphere obtained by gluing three
self-folded triangles to respective sides of an ordinary triangle.
The corresponding cluster quiver of $T$ can be obtained by gluing
four blocks of type II, which can be shown as the following Figure
7. In this figure, for $i=1, 2$ and $3$, $i$ and $i'$ denote the
corresponding vertices of two arcs in the same self-folded
triangles. Obviously it is a planar quiver and hence in this case
$g'=g=0$.

\begin{figure}[h] \centering

  \includegraphics*[202,444][289,532]{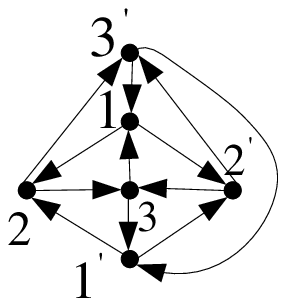}

Figure 7
 \end{figure}
 This completes the proof.
$\ \ \ \ \ \ \ \ \ \ \ \ \ \ \ \ \ \ \ \ \ \ \ \ \ \ \ \ \ \ \ \ \ \ \ \ \ \ \ \ \ \ \ \ \ \ \ \ \ \ \ \ \ \ \ \ \ \ \ \ \ \ \ \ \ \ \ \ \ \ \ \ \ \ \ \ \ \ \ \ \ \ \ \ \ \ \ \ \ \ \ \ \ \ \ \ \ \ \ \ \ \ \ \ \ \ \ \ \ \ \ \ \ \ \ \ \ \ \ \ \ \ \ \ \ \ \ \ \ \ \ \ \ \  \Box$

\vspace{5mm}
In order to prove Theorem \ref{thm1.1}(ii), we need
some preliminaries from \cite{GT} and \cite{M}.

Let us firstly introduce a class of graphs with arbitrary large
genus. For each positive integer $n$, the graph $R_{n}$ is
constructed as follows (see \cite{GT} Example 3.4.2):

There are $n+1$ concentric cycles which consists of $4n$ edges.
Additionally, there are $4n^{2}$ inner edges connecting the $n+1$
cycles to each other and $2n$ outer edges adjoining antipodal
vertices on the outermost cycle.

The following Figure 8 gives the example $R_{2}$ for such graph in
the case $n=2$.
\begin{figure}[h] \centering

  \includegraphics*[167, 409][302, 545]{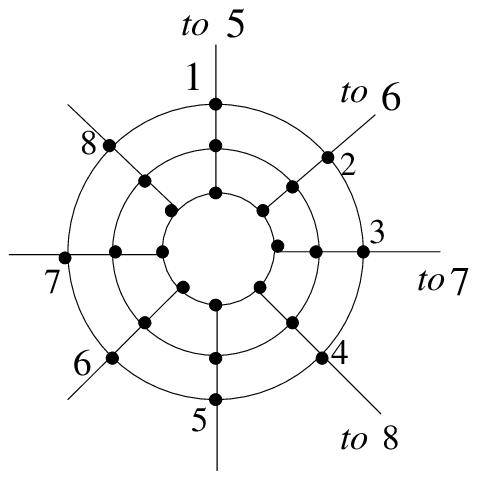}

Figure 8
\end{figure}

It was shown in \cite{GT} that $R_{n}$ is of genus $n$.

Secondly, we recall some basic notions and facts in \cite{M}.

The classification theorem of compact(or closed) surfaces (see
Theorem 5.1 of \cite{M} Chapter 1) asserts that any compact surface
could be either homeomorphic to a sphere, or to a connected sum of
tori, or to a connected sum of projective planes. The compact
surfaces can be considered as the quotient space of a polygon with
the directed edges identified in pairs. There is a convenient way to
indicate which paired edges are to be identified in such a polygon.
We give a letter (for example $a, b, c, \cdots$) to each paired
edges such that different pairs receive different letters. Starting
at a definite vertex, we traverse the boundary of the polygon in
either a clockwise or counter clockwise fashion. If the arrow on an
edge points in the same traversing direction, then we will put no
exponent (or the exponent $+1$) on the letter for that edge;
otherwise, we will write the letter for that edge with the exponent
$-1$. For example, the identifications of the following hexagon in
Figure 9 can be indicated by the symbols as
$a_{1}a_{1}a_{2}a_{2}^{-1}a_{3}a_{3}^{-1}$.

\begin{figure}[h] \centering

  \includegraphics*[214, 432][309, 532]{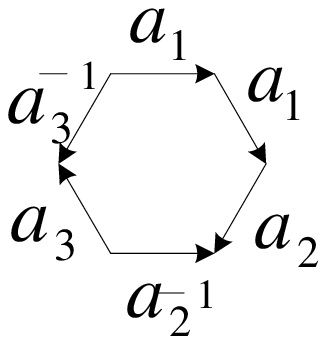}

Figure 9
\end{figure}

The symbols corresponding to the surfaces mentioned in the
classification theorem are as follows (see \cite{M} $\S 5$):

(1) The sphere: $aa^{-1}$.

(2) The connected sum of $n$ tori:
$$a_{1}b_{1}a_{1}^{-1}b_{1}^{-1}a_{2}b_{2}a_{2}^{-1}b_{2}^{-1}\cdots
a_{n}b_{n}a_{n}^{-1}b_{n}^{-1}.$$

(3) The connected sum of $n$ projective planes:
$$a_{1}a_{1}a_{2}a_{2}\cdots a_{n}a_{n}.$$

Given a polygon, if the letter designating a certain pair of edges
occurs with both exponents $+1$ and $-1$ in the symbol, then this
pair of edges is said to be the {\bf first kind}; otherwise the pair
is said to be the {\bf second kind} (\cite{M}). From the proof of
Theorem 5.1 in \cite{M}, we know that if all the pairs of edges are
of the first kind, then the resulting surface is oriented; if there
exists a second kind pair, then the resulting surface is
non-oriented. Moreover, since the first kind pair of adjacent edges
can be eliminated, the resulting surface of a $4n-$gon with pairs
all of the first kind is an oriented surface with genus at most $n$.

For the preparation of the proof of Theorem \ref{thm1.1}(ii), we
first prove the following lemma.
\begin{lemma}\label{2}
For an arbitrary non-negative integer $n$, there always exists a
block-decomposable cluster quiver  $T_n$ such that the genus
$g(T_n)$ of $T_n$ satisfies that $g(T_n)\geq n$.
\end{lemma}
\begin{proof}
Given a graph $R_{n}$ as above, label the $n+1$ cycles from
innermost to outermost by $1$ to $n+1$. For each $i\in \{1, 2,\cdots
, n\}$, there are $4n$ rectangles  between the $ith$ cycle and the
$(i+1)th$ cycle. For the outermost cycle, there exist $2n$
rectangles between the $(n+1)th$ cycle and itself. Two rectangles
are said to be {\bf neighbored} if they share a common edge;
otherwise, they are said to be {\bf distant}. It is easy to observe
that there are $4n^{2}+2n$ rectangles in $R_{n}$. Given any
rectangle $A$ in $R_{n}$, we first choose four rectangles distant to
$A$ but having a common vertex with $A$; for these four rectangles,
we do the same thing as previous step; continuing this process, we
will obtain a maximal set of mutually distant rectangles. This is
denoted by $\mathcal S$. This set contains $2n^{2}+n$ rectangles.
The other $2n^{2}+n$ rectangles form another maximal set of mutually
distant rectangles. This is denoted by $\mathcal T$. Trivially,
these two sets $\mathcal S$ and $\mathcal T$ are independent of the
choice of the original rectangle $A$. Consider the set $\mathcal S$,
each rectangle in $\mathcal S$ can be obtained by gluing four blocks
of type II as the following Figure 10:

\begin{figure}[h] \centering

  \includegraphics*[236, 430][302, 475]{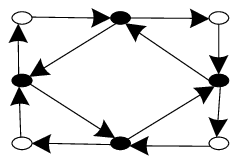}

 Figure 10
\end{figure}

For the innermost cycle, there are $2n$ edges which do not lie in
any rectangles of $\mathcal S$. We can then use one block of type IV
to substitute each such edge. For all these $2n$ edges, we need $2n$
blocks of type IV.

In summary, we obtain a quiver $T_n$ by gluing $8n^{2}+4n$ blocks of
type II and $2n$ blocks of type IV. According to the construction of
$T_n$, obviously, $R_{n}$ is a subgraph of the underlying graph of
$T_n$. Therefore, the genus $g(T_n)\geq g(R_n)=n$. The quiver of the
following Figure 11 is the example of $T_n$ in the case $n=2$, where
the vertices labeled by the same numbers should be glued together.
\end{proof}

\begin{figure}[h] \centering

  \includegraphics*[171, 465][394, 662]{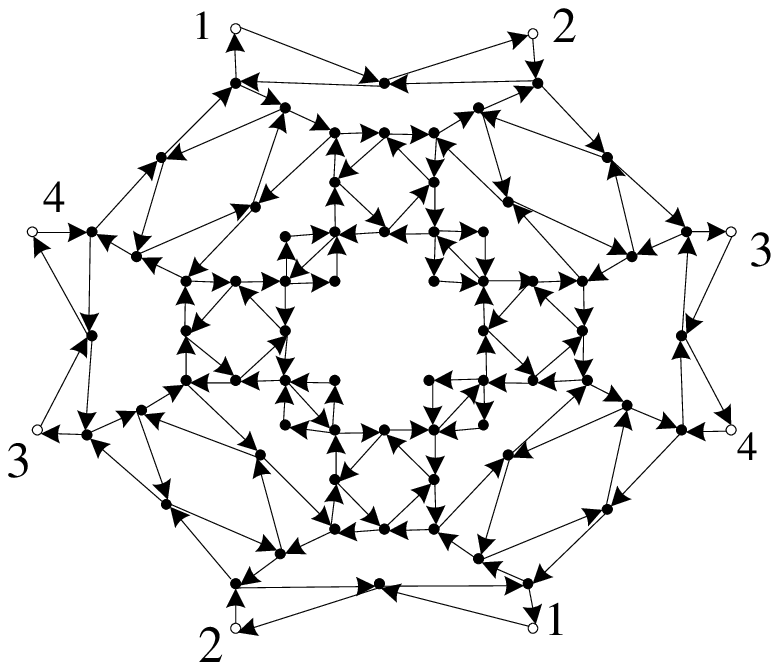}

 Figure 11
\end{figure}

Now, we can give the proof of part (ii) of Theorem \ref{thm1.1}. The
proof will be based on the fact that the quiver $T_{n}$ given in the
proof of Lemma \ref{2} can be obtained from a closed surface of
genus $n$.

{\em Proof of Theorem \ref{thm1.1}}(ii).~ By Lemma \ref{2},
$g(T_{n})\geq n$. It is easy to check that $T_{n}$ is a uniquely
block-decomposable quiver and hence $T_{n}$ can be uniquely encoded
by its corresponding triangulation, that is, blocks of type II are
encoded by puzzle pieces of first type (see the left graph in Figure
2(II)) and blocks of type IV are encoded by puzzle pieces of second
type (see the middle graph in Figure 2(II)). In order to draw
$T_{_{n}}$, we firstly draw a planar quiver $T_{n}'$ which has $4n$
unglued outlets. After gluing these $4n$ outlets in pairs, one
obtains $T_{n}$, where each pair consists of one outlet and its
opposite one. As example in the case $n=2$,  see Figure 11. Now, we
will construct a closed surface $S_{n}$ of genus $n$ and a
triangulation $P_{n}$ of $S_{n}$ such that its corresponding cluster
quiver is $T_{n}$.

We will chase $T_{n}$ from innermost to outermost. Blocks of type II
and type IV are encoded by puzzle pieces of first type and second
type respectively. For the outermost $4n$ oriented triangles in
$T_{n}^{'}$, we let each of them corresponds to a puzzle piece of
first type. Thus we can first obtain a $4n$-gon with triangulation.
Denote this $4n$-gon with triangulation as $S_{n}^{'}$. Then we can
obtain a closed oriented surface $S_{n}$ by identifying the edges of
$S_{n}^{'}$ in pairs and gluing all outermost vertices into one, and
then obtain a triangulation $P_{n}$ of $S_{n}$ such that its
corresponding quiver is exactly $T_{n}$. For the case $T_{2}$, its
corresponding $S_{2}^{'}$ can be shown as the following Figure 12.
To obtain $S_{2}$ and $P_{2}$, one only need to glue the edges
labeled by the same number in pairs and to glue all outermost 8
vertices into one.
\begin{figure}[h] \centering

  \includegraphics*[201, 221][389, 422]{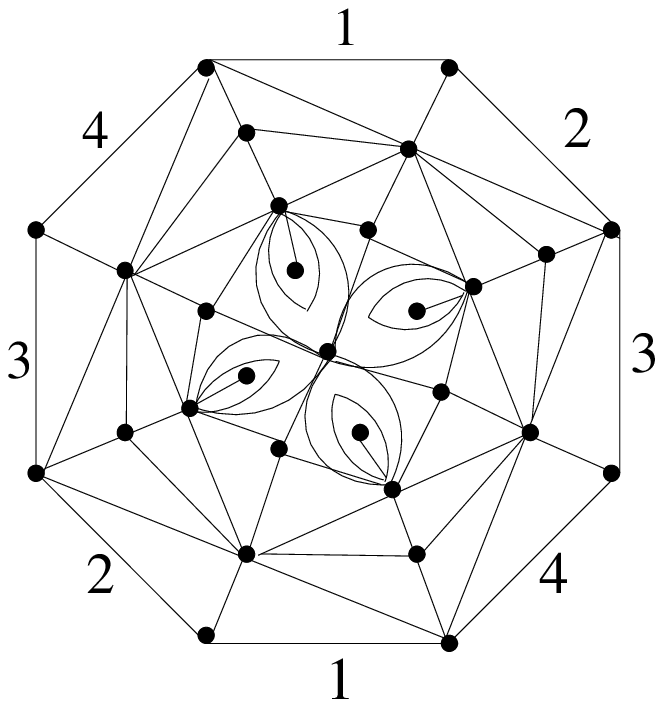}

 Figure 12
\end{figure}

By the proof of the classification theorem of compact surfaces in
\cite{M}, the genus of $S_{n}$ is at most $n$.

Since $T_{n}$ is obtained from a triangulation of $S_{n}$, by
Theorem \ref{thm1.1}(i), $g(T_{n})\leq n$.

On the other hand, by Lemma \ref{2}, $g(T_{n})\geq n$. Hence,
$g(T_{n})=n$.

For the genus $g(S_{n})$ of $S_{n}$, since $n=g(T_{n})\leq
g(S_{n})\leq n$, we also have $g(S_{n})=n$.

Then Theorem \ref{thm1.1}(ii) easily follows from the fact that all closed oriented
surfaces with the same genus are homeomorphic. $\ \ \ \ \ \ \ \ \ \
\ \ \ \ \ \ \ \ \ \ \ \ \ \ \ \ \ \ \ \ \ \ \ \ \ \ \ \ \ \ \ \ \ \
 \ \ \ \ \ \ \ \ \ \ \ \ \ \ \ \ \ \ \ \ \ \ \ \ \ \ \ \ \ \ \ \ \ \
 \ \ \ \ \ \ \ \ \ \ \ \ \ \ \ \ \ \ \ \ \ \ \ \ \ \ \ \ \ \ \ \ \
 \ \ \ \ \ \ \ \ \ \ \ \ \ \ \ \ \ \ \ \ \Box$
\\

\subsection{Applications and further problems}
As an application of Theorem \ref{thm1.1}(i), we firstly give two
corollaries.

\begin{corollary}\label{cor5.2}
Let S be a surface of genus $0$ and $M$ be a set of marked points of
$S$. Given any triangulation $T$ of $(S, M)$, suppose $Q$ is the
associated cluster quiver from $T$, then all quivers in the
mutation-equivalent class of $Q$ are of genus $0$.
\end{corollary}
Besides the cluster quivers of type E in Section 3.1, this corollary
gives another class of cluster quivers of finite mutation type whose
genuses are invariant under mutation.
\begin{corollary}
Let $S$ be a surface of genus $g$ with $M$ the set of marked points.
For any triangulation $T$ of $(S, M)$, let $Q$ be its corresponding
quiver and let $Q'$ be another cluster quiver of genus $g'$ such
that $g'>g$, then $Q$ and $Q'$ are not mutation-equivalent.
\end{corollary}
\begin{proof}
According to Proposition 12.3 in \cite{fosth}, all quivers in the
mutation-equivalent class of $Q$ are the corresponding quivers of
some triangulations of $(S, M)$. Hence, by Theorem \ref{thm1.1}(i),
the genuses of these quivers are not greater than $g$.
 Hence $Q'$ is not in the mutation-equivalent class of $Q$, that is, $Q$ and
$Q'$ are not mutation-equivalent.
\end{proof}
This corollary gives us a necessary condition for two
quivers, one is from triangulation of a surface and the other is non-planar, with the same numbers of vertices,  to be mutation-equivalent.

\begin{remark}\label{rem3.5}
An easy calculation shows that the number of marked points on the
closed surface $S_{n}$ in the proof of Theorem \ref{thm1.1}(ii) is
$4n^{2}+2n+2$. For example in the case $n=2$, one can easily see
that there are $22$ marked points on $S_{2}$, here the outermost $8$
marked points in $S_{2}^{'}$(see Figure 12) are glued into one.
\end{remark}

Theorem \ref{thm1.1}(ii) tells us that given a closed surface $S$ of
genus $n$, the upper bound of genuses of quivers from triangulations
of $S$ given in part (i) of Theorem \ref{thm1.1} can be reached.

On the other hand, the lower bound $0$ of genuses can also be reached, that is, given any closed oriented surface $S$ with genus $n$, there
always exists a triangulation $T$ of $S$ such that the corresponding cluster
quiver $Q$ of $T$ is planar.

In fact, if the closed surface is a sphere, it obviously holds by
Corollary \ref{cor5.2}. If the closed surface $S$ is of genus
$n(\geq 1)$, then it is homeomorphic to the connected sum of
$n$-tori. In this case, the symbol of its corresponding polygon is
$a_{1}b_{1}a_{1}^{-1}b_{1}^{-1}a_{2}b_{2}a_{2}^{-1}b_{2}^{-1}\cdots
a_{n}b_{n}a_{n}^{-1}b_{n}^{-1}$. A triangulation $T$ of $S$ with two
punctures is as follows in Figure 13. For this triangulation, the
outer $4n$ vertices in fact come from the same puncture and the only
inner vertex is the other puncture. One can easily check that the
corresponding cluster quiver $Q$ of $T$ is planar.

\begin{figure}[h] \centering

  \includegraphics*[176, 521][275, 616]{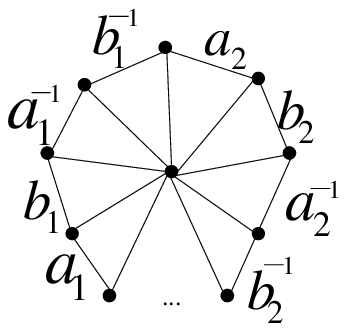}

Figure 13
\end{figure}
Restricting the discussion to torus, we can reach the
following conclusion:
\begin{proposition}\label{prop3.8} For a given cluster quiver $Q$ from
the torus $S$ with $p$ punctures, there exists at least one planar
quiver in the mutation-equivalent class of $Q$.
\end{proposition}
\begin{proof}
 According to
Proposition 12.3 in \cite{fosth}, the corresponding quivers from all
triangulations of $S$ are mutually mutation-equivalent. Hence, we
only need to find a triangulation $T$ of $S$ such that its
corresponding quiver is planar.

For the convenience of describing the desired
triangulation, we first restate how a torus is
constructed. Given two circles $C$ and $C'$, assume the radius of
$C$ is greater than that of $C'$. Let the center of $C'$ run along
$C$ for one round, then a torus is built. The circle $C$ is called a
{\bf basic circle} for this torus.

 For the torus $S$ with $p$ punctures, we construct a
triangulation $T$ as follows:

For each puncture, construct a closed arc on $S$ perpendicular to
the basic circle such that its two endpoints are both coincided at
this puncture.  Thus, we have such $p$ arcs.
 These $p$ arcs cut down the torus
into $p$ pieces of cylinders. For each cylinder, drawing an arc
between two punctures, we obtain a rectangle. Moreover, we draw a
diagonal in this rectangle. The corresponding quiver from such a
rectangle with diagonal is shown as in the following Figure 14.

\begin{figure}[h] \centering

  \includegraphics*[250,492][320,544]{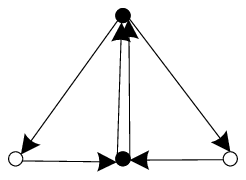}

 Figure 14
 \end{figure}

All such $p$ rectangles with diagonal are arranged continuously
together to form a graph. The quiver $Q$ of $T$ is obtained by
gluing $p$ pieces of such quivers along the outlets. Obviously, it
is a planar quiver. For example, in the case $p=3$, the
triangulation can be shown as in Figure 15(1), where the numbers
$1,\cdots,9$ label the arcs, and its corresponding cluster quiver is
shown as in Figure 15(2).

\begin{figure}[h] \centering

  \includegraphics*[111,431][441,534]{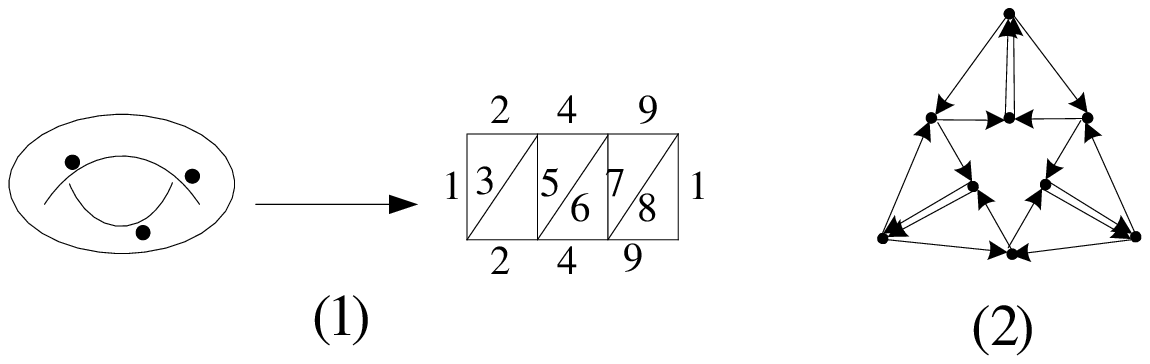}

 Figure 15
\end{figure}
\end{proof}
Since both of the upper bound and lower bound can reach for genuses
of cluster quivers from closed surface, based on Theorem
\ref{thm1.1} and Proposition \ref{prop3.8}, we propose the further
interesting problems as follows:
\begin{problem}\label{prb3.7}
For any closed surface $S$ with genus $n$ and $0\leq i\leq n$, does
there exist a certain number of punctures and an ideal triangulation
$T^{(i)}$ of $S$ such that the corresponding cluster quiver $Q_{i}$
from $T^{(i)}$ is of genus $i$?
\end{problem}

\begin{problem}\label{prb3.8}
Given a closed surface $S$ of genus $n$. Find out the minimal number
of punctures on $S$ with the property that there exists an ideal
triangulation $T$ of $S$ such that the corresponding cluster quiver
$Q_n$ of $T$ is exactly of genus $n$.
\end{problem}
For the case of torus, we know at least one planar quiver in each
mutation-equivalent class according to Proposition \ref{prop3.8}.
Hence for a given number of punctures we can check the corresponding
mutation-equivalent class of this planar quiver by Keller's quiver mutation in Java
\cite{k}. Since the genus of a quiver has nothing to do with the
orientations of arrows, we can choose the mutation-equivalent class
under graph isomorphism when doing quiver mutation in Java.

For the cases $p=1$ and $p=2$, all quivers in their two
mutation-equivalent classes are planar. When $p=3$, there exists
exactly one quiver of genus $1$ (see Figure 16) in the mutation
class. Therefore, the answer of Problem \ref{prb3.8} for the case of
torus is $p=3$, which is much smaller than the number $4\times
1^{2}+2\times 1+2=8$ of punctures given in Remark \ref{rem3.5} when
constructing $T_{1}$ from torus.
\begin{figure}[h] \centering

  \includegraphics*[193,471][272,516]{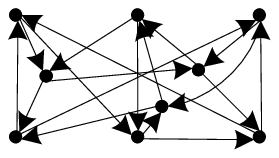}

 Figure 16
\end{figure}

{\bf Acknowledgements.}~~ The authors are grateful to the referees
for their important suggestions in improving the quality of
this paper. The authors also warmly thank Andrei Zelevinsky for his
 helpful comments and suggestions. We dedicate this paper to his
memory. 

 This work was supported by the National Natural Science Foundation of China (No.11271318, No.11171296 and No.J1210038) and the Specialized Research Fund for the Doctoral Program of Higher Education of China (No.20110101110010) and the Zhejiang Provincial Natural Science Foundation of China (No.LZ13A010001).

\end{document}